\documentclass[11pt,twoside]{amsart}
\usepackage[latin1]{inputenc}
\usepackage[T1]{fontenc}
\usepackage{mathtools}
\usepackage{graphicx}
\usepackage{tikz}
\usetikzlibrary{chains}
\usepackage{tcolorbox}
\usepackage{mathabx}

\usepackage[latin1]{inputenc}
\usepackage{amsmath}
\usepackage{amsthm}
\usepackage{amssymb}

\usepackage[all]{xy}
\usepackage{hyperref}
\newtheorem{thm}{Theorem}[section]

\newtheorem{prop}[thm]{Proposition}

\newtheorem{lem}[thm]{Lemma}
\newtheorem{cor}[thm]{Corollary}

\theoremstyle{definition}
\newtheorem{definition}[thm]{Definition}
\newtheorem{remark}[thm]{Remark}
\newtheorem{ex}[thm]{Example}
\numberwithin{equation}{section}

\newtheorem*{claim}{Claim}

\DeclareMathOperator{\End}{\mathrm{End}}

\renewcommand{\to}{\xymatrix@1@=15pt{\ar[r]&}}
\renewcommand{\rightarrow}{\xymatrix@1@=15pt{\ar[r]&}}
\renewcommand{\leftarrow}{\xymatrix@1@=15pt{&\ar[l]}}
\renewcommand{\mapsto}{\xymatrix@1@=15pt{\ar@{|->}[r]&}}
\renewcommand{\twoheadrightarrow}{\xymatrix@1@=18pt{\ar@{->>}[r]&}}
\renewcommand{\hookrightarrow}{\xymatrix@1@=15pt{\ar@{^(->}[r]&}}
\newcommand{\congpf}{\xymatrix@L=0.6ex@1@=15pt{\ar[r]^-\sim&}}
\renewcommand{\cong}{\simeq}
\newcommand{\LLV}{\mathfrak{g}_{\mathrm{tot}}(X)}
\newcommand{\Spin}{\underline{\mathrm{Spin}}(H^2(X,\mathbb{Q}),q)}
\newcommand{\SO}{\underline{\mathrm{SO}}(H^2(X,\mathbb{Q}),q)}

\usepackage{comment}
\excludecomment{comment}
\begin{document}

\title[]{The Looijenga--Lunts--Verbitsky algebra and Verbitsky's Theorem}

\author[A.\ Bottini]{Alessio Bottini}

\address{Dipartimento di Matematica, Universit\`a di Roma Tor Vergata, Via della Ricerca Scientifica 1, 00133, Roma, Italia}
\address{
Universit\'e Paris-Saclay, CNRS, Laboratoire de Math\'ematiques d'Orsay, Rue Michel Magat, Bat.~307, 91405 Orsay, France}
\email{bottini@mat.uniroma2.it}

\begin{abstract} \noindent
In these notes we review some basic facts about the LLV Lie algebra. It is a rational Lie algebra, introduced by Looijenga--Lunts and Verbitsky, acting on the rational cohomology of a compact K\"ahler manifold. We study its structure and describe one irreducible component of the rational cohomology in the case of a compact hyperk\"ahler manifold.

\vspace{-2mm}
\end{abstract}

\maketitle
{\let\thefootnote\relax\footnotetext{This review was prepared in the context of the
seminar organized by the ERC Synergy Grant ERC-2020-SyG-854361-HyperK.
The talk was delivered on April 23, 2021.}}
\marginpar{}

\section{Introduction}
\subsection{}
Let $V=\bigoplus_{k \in \mathbb{Z}} V_k $ be a finite dimensional graded vector space over a field $\mathbb{F}$ of characteristic $0$, and denote by $h$ the operator: 
\[ h|_{V_k}=k \cdot \mathrm{id}.\]

\begin{definition}
Let $e \colon V \rightarrow V$ be a degree $2$ endomorphism. We say $e$ has the \textit{Lefschetz property} if
\[ e^k \colon V_{-k}  \rightarrow V_k\]
is an isomorphism.
\end{definition}

\begin{remark}
The degree two operators with the Lefschetz property form a Zariski open subset of $\End_2(V)$. 
\end{remark}

\begin{thm}[Jacobson--Morozov, {{\cite[Theorem $3$]{Jac51}}}]\label{thm:JM}
An operator $e$ has the Lefschetz property if and only if there exists a unique degree $-2$ endomorphism $f\colon V \rightarrow V$ such that 
\[ [e,f]=h.\]
Moreover, if $L \subset \End(V)$ is a semisimple Lie subalgebra and $e,h \in L$, then $f \in L$. 
\end{thm}
We say that the triple $(e,h,f)$ is an $\mathfrak{sl}_2$-triple, the reason is that we can define a representation $\mathfrak{sl}_2(\mathbb{F}) \rightarrow \mathrm{End}(V) $ of the Lie algebra $\mathfrak{sl}_2(\mathbb{F})$ on the vector space $V$ as follows
\begin{equation*}
\begin{pmatrix}
0 & 1 \\
0 & 0 
\end{pmatrix} \mapsto e, \qquad
\begin{pmatrix}
1 & 0  \\
0 & -1 
\end{pmatrix}  \mapsto h, \qquad
\begin{pmatrix}
0 & 0  \\
1 & 0
\end{pmatrix} \mapsto f.
\end{equation*}

In the rest of these notes, we will mostly be interested in the graded rational vector space $V=H^{*}(X,\mathbb{Q})[N]$, where $X$ is a compact K\"ahler manifold of dimension $N$. Here $[m]$ indicates the shift by $m$, so that $V_0=H^N(X,\mathbb{Q})$. To any class $a \in H^2(X,\mathbb Q)$ we can associate the operator in cohomology obtained by taking cup product
\[ e_{a} \colon H^*(X,\mathbb Q) \rightarrow H^*(X,\mathbb Q), \quad  \omega \mapsto a. \omega.\]
The operator $h$ becomes
\[  h|_{H^{k}(X,\mathbb Q)} \coloneqq (k-N)\mathrm{id}.\]
From Theorem \ref{thm:JM} we see that if $e_{a}$ has the Lefschetz property (for example if $a$ is a K\"{a}hler class), there is an operator $f_{a}$ of degree $-2$ that makes $(e_{a},h,f_{a})$ an $\mathfrak{sl}_2$-triple. Moreover, the map 
$$ f\colon H^2(X,\mathbb{Q}) \dashrightarrow \End_{-2}(H^*(X,\mathbb{Q})),$$
that sends $a$ to the operator $f_{a}$ is defined on a Zariski open subset and rational.

\begin{remark}
If $a \in H^{1,1}(X,\mathbb Q)$ is K\"ahler, it follows from standard Hodge theory that everything can be defined at the level of forms. The dual operator is $f_{a}=*^{-1}e_{a}*$, where $*$ is the Hodge star operator. The $\mathfrak{sl}_2$-action preserves the harmonic forms, so it induces an action on cohomology. 
\end{remark}

\begin{definition}[{{\cite{LL97:LooijengaLunts, Verbitsky:thesis}}}]
Let $X$ be a compact K\"ahler manifold. The \textit{total Lie algebra} $\LLV$ of $X$ is the Lie algebra generated by the $\mathfrak{sl}_2$-triples 
$$  (e_{a},h,f_{a}), $$
where $a \in H^2(X,\mathbb Q)$ is a class with the Lefschetz property.
\end{definition}

The following is a general result about this Lie algebra for compact K\"ahler manifolds. Denote by $\phi$ the pairing on $H^*(X,\mathbb{C})$ given by 
\[ \phi(\alpha,\beta) = (-1)^q \int_X{\alpha.\beta}, \]
if $\alpha $ has degree $N+2q$ or $N+2q+1$.

\begin{prop}[{{\cite[Proposition $1.6$]{LL97:LooijengaLunts}}}]\label{LL:ss}
The Lie algebra $\LLV$ is semisimple and preserves $\phi$ infinitesimally. Moreover, the degree-$0$ part $\LLV_0$ is reductive. 
\end{prop}

\subsection{}
Now let $X$ be a compact hyperk\"ahler manifold of complex dimension $2n$. In this case, the Lie algebra $\LLV$ is also called the Looijenga-Lunts-Verbitsky Lie algebra. It is well known that for each hyperk\"ahler metric $g$ on $X$ we get an action of the quaternion algebra $\mathbb{H}$ on the real tangent bundle $TX$. This means that we have three complex structures $I,J,K$ such that
\begin{equation}\label{eq:quaternions}
    IJ=-JI=K.
\end{equation}
To each of these complex structures we can associate a K\"ahler form $\omega_I \coloneqq g(I(-),-), \omega_J \coloneqq (J(-),-),\omega_K \coloneqq g(K(-),-)$ and a holomorphic symplectic form $\sigma_I=\omega_J + i \omega_K, \sigma_J=\omega_K +i \omega_I, \sigma_K=\omega_I+i\omega_J$. 

\begin{definition}
The \textit{charateristic $3$-plane} $F(g)$ of the metric $g$ is \[ F(g) \coloneqq \langle [\omega_I],[\omega_J],[\omega_K] \rangle = \langle [\omega_I],[\Re \sigma_I],[\Im \sigma_I] \rangle \subset H^2(X,\mathbb R).\]
\end{definition}

\begin{definition}[{{\cite{Ver90:VerbitskySO(5)}}}]
Denote by $\mathfrak{g}_g \subset \mathrm{End}(H^*(X,\mathbb R))$ the Lie algebra generated by the $\mathfrak{sl}_2$-triples $(e_{a},h,f_{a})$ where $a \in F(g)$. 
\end{definition}

\begin{remark}
This Lie algebra is generated by the three $\mathfrak{sl}_2$-triples associated to the classes $[\omega_I],[\omega_J],[\omega_K]$. Indeed, from the discussion in the following section we will see that the subalgebra generated by these three $\mathfrak{sl}_2$-triples is semisimple. From the Jacobson-Morozov Theorem and the linearity of $e\colon H^2(X,\mathbb{R}) \rightarrow \End(H^*(X,\mathbb{R}))$ we conclude that it contains every $\mathfrak{sl}_2$-triple $(e_{a},h,f_{a})$ with $a \in F(g)$. 
\end{remark}

\section{The algebra associated to a metric}
\subsection{}
In this section we study the smaller algebra $\mathfrak{g}_g$ and its action on cohomology. These results are due to Verbitsky \cite{Ver90:VerbitskySO(5)}, see also \cite{GHJ03:CY_manifolds}.

We start with a general algebraic construction. Let $\mathbb{H}$ be the quaternion algebra. As a real vector space it is generated by $1,I,J,K$, where $I,J,K$ satisfy the relations \eqref{eq:quaternions}. We denote by $\mathbb{H}_0$ the pure quaternions, i.e. the linear combinations of $I,J,K$. 

Let $V$ be a left $\mathbb{H}$-module, equipped with an inner product
\[ \langle -,- \rangle \colon V \times V \rightarrow \mathbb{R},\] and assume that $I,J,K$ act on $V$ via isometries. The $\mathbb{H}$-action gives three complex structures $I,J,K$ on $V$, satisfying the relations \eqref{eq:quaternions}. Consider the forms
\[ \omega_I=\langle I(-),- \rangle,\]
\[\omega_J=\langle J(-),- \rangle, \] \[\omega_K=\langle K(-),- \rangle \] and the holomorphic symplectic forms $\sigma_I=\omega_J + i \omega_K, \sigma_J=\omega_K +i \omega_I, \sigma_K=\omega_I+i\omega_J$. 

\begin{remark}
Note that the operators $e_{\lambda}$ for $\lambda=\omega_I,\omega_J,\omega_K$ have the Lefschetz property; the dual operator is given by $f_{\lambda}=*^{-1} e_{\lambda} * $, where $*$ is the Hodge star operator on $\Lambda^{\bullet} V^*$ induced by the inner product.
\end{remark}

\begin{definition}
Let $\mathfrak{g}(V) \subset \mathrm{End}(\bigwedge^{\bullet} V^*)$ be the Lie algebra generated by the $\mathfrak{sl}_2$-triples 
\[ (e_{\lambda},h,f_{\lambda})_{\lambda=\omega_I,\omega_J,\omega_K}, \] where $h$ is the shifted degree operator.
\end{definition}

In particular, this definition makes sense for the rank one module $\mathbb{H}$ equipped with the standard inner product. This gives a Lie algebra $\mathfrak{g}(\mathbb{H}) \subset \mathrm{End}(\bigwedge^{\bullet} \mathbb{H}^*).$ We denote by $\mathfrak{g}(\mathbb{H})_0$ the degree-$0$ component of $\mathfrak{g}(\mathbb{H})$ (here the degree is viewed as an endomorphism of the graded vector space). It is a Lie subalgebra, and we denote it by $\mathfrak{g}(\mathbb{H})'_0 \coloneqq [\mathfrak{g}(\mathbb{H})_0,\mathfrak{g}(\mathbb{H})_0]$ its derived Lie algebra. 

\begin{thm}\label{valg}
With the above notation we have the following.
\begin{enumerate}
    \item There is a natural isomorphism $\mathfrak{g}(V) \cong \mathfrak{g}(\mathbb{H}).$
    \item There is an isomorphism $\mathfrak{g}(\mathbb{H}) \cong \mathfrak{so}(4,1)$.
    \item The algebra decomposes with respect to the degree as \[ \mathfrak{g}(\mathbb{H})=\mathfrak{g}(\mathbb{H})_{-2} \oplus \mathfrak{g}(\mathbb{H})_0 \oplus \mathfrak{g}(\mathbb{H})_{2}.\] 
    Furthermore, $\mathfrak{g}(\mathbb{H})_{\pm 2} \cong \mathbb{H}_0$ as Lie algebras, and $\mathfrak{g}(\mathbb{H})_0=\mathfrak{g}(\mathbb{H})'_0 \oplus \mathbb{R}h$ 
    with $\mathfrak{g}(\mathbb{H})'_0 \cong \mathbb{H}_0$; this last isomorphism is compatible with the actions on $\bigwedge ^{\bullet} V^*$. 
\end{enumerate}
\end{thm}

\begin{proof}
$(1)$ Since $\langle -,- \rangle$ is $\mathbb{H}$-invariant, we can find an orthogonal decomposition 
\[ V = \mathbb{H} \oplus \dots \oplus \mathbb{H}. \]
Taking exterior powers we get $\bigwedge ^{\bullet} V^* = \bigwedge ^{\bullet} \mathbb{H}^* \otimes \dots \otimes \bigwedge ^{\bullet} \mathbb{H}^*$. This gives an injective map $\mathfrak{g}(\mathbb{H}) \rightarrow \End(\bigwedge ^{\bullet} V^*)$, given by the natural tensor product representation. It is a direct check that the image of this morphism is exactly the algebra $\mathfrak{g}(V)$. 

$(2)$ Consider the subrepresentation $W \subset \bigwedge ^{\bullet} \mathbb{H}^* $ given by
\[ W = \bigwedge ^{0} \mathbb{H}^* \oplus \langle \omega_I,\omega_J,\omega_K \rangle \oplus \bigwedge^{4}\mathbb{H}^*. \]
We equip it with the quadratic form given by setting $\bigwedge ^{0} \mathbb{H}^* \oplus \bigwedge^{4}\mathbb{H}^*$ to be a hyperbolic plane, orthogonal to the $3$-plane, and $\{\omega_I,\omega_J,\omega_K\}$ to be an orthonormal basis of the $3$-plane. By a direct computation we can see that the action of $\mathfrak{g}(\mathbb{H})$ respects infinitesimally this quadratic form. This gives a map 
\begin{equation}\label{eq:isomSO}
    \mathfrak{g}(\mathbb{H}) \rightarrow \mathfrak{so}(W) \cong \mathfrak{so}(4,1),
\end{equation} 
that we next show to be an isomorphism. 

Since $W $ has dimension $5$ the Lie algebra $\mathfrak{so}(W)$ has dimension $10$. Now consider the following $10$ elements of $\mathfrak{g}(\mathbb{H})$: 
\[ h,e_I,e_J,e_K,f_I,f_J,f_K,K_{IJ},K_{IK},K_{JK}, \]
where $K_{IJ} \coloneqq [e_I,f_J], K_{IK}=[e_I,f_K]$ and $K_{JK}=[e_J,f_K]$. Verbitsky \cite{Ver90:VerbitskySO(5)} showed that $K_{IJ}$ acts like the Weil operator associated with the Hodge structure on $\bigwedge^{\bullet}\mathbb{H}^*$ given by $K$, and similarly $K_{JK}$ and $K_{IK}$. This means that on a $(p,q)$ form with respect to $K$ it acts as multiplication by $i(p-q)$. It follows that the ten operators above are linearly independent over $W$, hence the map is surjective. Moreover they generate $\mathfrak{g}(\mathbb{H})$ as a vector space. Indeed, they generate $\mathfrak{g}(\mathbb{H})$ as a Lie algebra, and one has the following relations (see \cite{Ver90:VerbitskySO(5)}): 
\begin{align*}
    [K_{\lambda,\mu},K_{\mu,\nu}]=K_{\lambda,\nu}, & \quad [K_{\lambda,\mu},h] = 0, \\
    [K_{\lambda,\mu},e_{\mu}]=2e_{\lambda},  & \quad [K_{\lambda,\mu},f_{\mu}]=2f_{\lambda}, \\
    [K_{\lambda,\mu},e_{\nu}]=0,  & \quad [K_{\lambda,\mu},f_{\nu}]=0,
\end{align*}
where $\lambda,\mu,\nu \in \{I,J,K\}$ and $\nu \neq \lambda, \nu \neq \mu$. This implies that they are a basis of $\mathfrak{g}(\mathbb{H})$, hence the map \eqref{eq:isomSO} is an isomorphism.

Point $(3)$ follows using this explicit basis. Indeed we have 
\begin{align*}
\mathfrak{g}(\mathbb{H})_{-2} &= \langle f_I,f_J,f_K \rangle, \quad
\mathfrak{g}(\mathbb{H})_{2} = \langle e_I,e_J,e_K \rangle, \  \textrm{and} \\
\mathfrak{g}(\mathbb{H})_{0} & = \langle K_{IJ},K_{JK},K_{IK} \rangle \oplus \mathbb{R}h. \\
\end{align*}
In particular we have 
\begin{align*}
    \mathfrak{g}(\mathbb{H})'_0 & \xrightarrow{\sim} \mathbb{H}_0, \\
    K_{IJ} & \mapsto K,\\
    K_{JK} & \mapsto I,\\
    K_{IK} & \mapsto J.
 \end{align*}
Since $I,J,K \in \mathbb{H}_0$ act on $\bigwedge^{\bullet}\mathbb{H}^*$ as Weil operators for the corresponding complex structures on $\mathbb{H}$, the isomorphism is compatible with the actions.
\end{proof}

Now we can compute the Lie algebra $\mathfrak{g}_g$. As above we denote by $(\mathfrak{g}_g)_0$ the degree-$0$ part, and by $(\mathfrak{g}_g)'_0 \coloneqq [(\mathfrak{g}_g)_0,(\mathfrak{g}_g)_0]$ its derived Lie algebra.

\begin{prop}
Let $(X,g)$ be a hyperk\"ahler manifold with a fixed hyperk\"ahler metric. 
\begin{enumerate}
    \item There is a natural isomorphism of graded Lie algebras $\mathfrak{g}_g \cong \mathfrak{g}(\mathbb{H})$. In particular $\mathfrak{g}_g \cong \mathfrak{so}(4,1)$.
    \item The semisimple part $(\mathfrak{g}_g)'_0$ acts on $H^*(X,\mathbb R)$ via derivations. 
\end{enumerate}
\end{prop}

\begin{proof}
$(1)$. Consider the Lie subalgebra $\hat{\mathfrak{g}}_{g} \subset \End(\Omega^{\bullet}_X)$, generated by the $\mathfrak{sl_2}$-triples $(e_a,h,f_a)$ with $a \in F(g)$, at the level of forms (in particular $f_a=*^{-1}e_a*$). From the previous proposition, we see that for every point $x \in X$ there is an inclusion $\mathfrak{g}(\mathbb{H})\  \hookrightarrow \End(\Omega^{\bullet}_{X,x})$. This gives an inclusion 
$\mathfrak{g}(\mathbb{H})\  \hookrightarrow \prod_{x \in X} \End(\Omega^{\bullet}_{X,x}) $. It follows from the definitions that the two algebras $\mathfrak{g}(\mathbb{H})$ and $\hat{\mathfrak{g}}_{g}$ are equal as subalgebras of $\prod_{x \in X} \End(\Omega^{\bullet}_{X,x})$.

Since the metric $g$ is fixed, the $\mathfrak{sl}_2$-triples $(e_a,h,f_a)$ preserve the harmonic forms $\mathcal{H}^*(X)$, and so does $\hat{\mathfrak{g}}_{g}$. Since $\mathcal{H}^*(X) \cong H^*(X,\mathbb{R})$ we get a morphism 
\[ \mathfrak{g}(\mathbb{H}) \cong \hat{\mathfrak{g}}_{g} \rightarrow \mathfrak{g}_g. \]
This map is surjective, because the image contains the $\mathfrak{sl}_2$-triples that generate $\mathfrak{g}_g$. Moreover, by explicit computations similar to the proof of the previous proposition, we can see that $\dim \mathfrak{g}_g \geq 10 $. Hence the map is an isomorphism. 

$(2)$. From the previous proposition we have an isomorphism compatible with the actions on cohomology
\[(\mathfrak{g}_g)'_0 \cong \mathfrak{g}(\mathbb{H})'_0 \cong \mathbb{H}_0.\]
Hence, it suffices to prove the statement for the action of $I,J,K$. Each of them gives a complex structure, and acts as the Weil operator on the associated Hodge decomposition. So, the action on $(p,q)$ forms is given by multiplication by $i(p-q)$, which is a derivation. 
\end{proof}

\section{The total Lie algebra}
The goal of this section is to prove the following result due to Looijenga and Lunts \cite[Proposition $4.5$]{LL97:LooijengaLunts} and Verbitsky \cite[Theorem $1.6$]{Verbitsky:thesis}.

\begin{thm}\label{thm:LL}
Let $X$ be a hyperk\"ahler manifold. With the above notation we have the following.
\begin{enumerate}
    \item The total Lie algebra $\LLV$ lives only in degrees $-2,0,2$, so it decomposes as: 
    \[ \LLV=\LLV_{-2} \oplus \LLV_0 \oplus \LLV_2. \]
    \item There are canonical isomorphisms $\LLV_{\pm 2} \cong H^2(X,\mathbb{Q})$.
    \item There is a decomposition $\LLV_0 = \LLV_0' \oplus \mathbb{Q}h$ with $\LLV_0' \cong \mathfrak{so}(H^2(X,\mathbb{Q}),q)$, where $q$ is the Beauville--Bogolomov--Fujiki quadratic form \cite{notes:BD}. Furthermore $\LLV_0'$ acts on $H^*(X,\mathbb{Q})$ by derivations.
\end{enumerate}
\end{thm}

The main geometric input in the proof is the following lemma.

\begin{lem}\label{lem:abelian}
If $X$ is a compact hyperk\"ahler manifold, then $[f_a,f_b]=0$ for every $a,b \in H^2(X,\mathbb{R})$ for which $f$ is defined. 
\end{lem}

The proof relies on the following fact. 

\begin{prop}\label{prop:3planes_open}
The set of charateristic $3$-planes is open in the Grassmannian of $3$-planes in $H^2(X,\mathbb{R})$. 
\end{prop}

In turn, this follows from a celebrated Theorem of Yau.

\begin{thm}[Yau]
Let $X$ be a hyperk\"ahler manifold, and let $I$ be a complex structure on $X$. If $\omega$ is a K\"ahler class, then there is a unique hyperk\"ahler metric $g$ such that 
$[\omega_I]=\omega.$
\end{thm}

\begin{proof}[Proof of Lemma \ref{lem:abelian}]
If we fix a hyperk\"ahler metric $g$ on $X$, then for every $a,b \in F(g)$ we have $[f_a,f_b]=0$. This holds already at the level of forms, using the definition $f_a=*^{-1}e_a*$ and the fact that $*$ depends only on the metric. Let $a \in H^2(X,\mathbb{R})$ be a class for which $f_a$ is defined. Since $f$ is rational, the condition $[f_a,f_b]=0$ is Zariski closed with respect to $b \in H^2(X,\mathbb{R})$. From Proposition \ref{prop:3planes_open} it follows that the set 
\[ \{ b \in H^2(X,\mathbb{R}) \mid a,b \in F(g) \textrm{ for some metric } g \} \]
is open. Since $[f_a,f_b]=0$ for every $b$ in this open set, we get $[f_a,f_b]=0$ for every $b$ where $f_b$ is defined.
\end{proof}

While the statement of Theorem \ref{thm:LL} is over $\mathbb{Q}$, we will give the proof over $\mathbb{R}$ following \cite{LL97:LooijengaLunts}. 

\begin{proof}[Proof of Theorem \ref{thm:LL}]
Consider the subspace \[V \coloneqq V_{-2} \oplus V_0 \oplus V_2 \subset \LLV,\] where $V_2$ is the abelian Lie subalgebra generated by $e_{a}$ with $a \in H^2(X,\mathbb{R})$, $V_{-2}$ is the abelian Lie subalgebra generated by the $f_{a}$ with $a \in H^2(X,\mathbb{R})$ where $f_a$ is defined, and $V_0$ is the Lie subalgebra generated by $[e_{a},f_{b}]$. To prove $(1)$ and $(2)$, it is enough to show that $V$ is a Lie subalgebra of $\LLV$. Indeed, since $\LLV$ is generated by elements contained in $V$ this would imply $V= \LLV$. Since $V_2$ and $V_{-2}$ are abelian, it suffices to show that
$[V_0,V_2] \subset V_2$ and $[V_0,V_{-2}] \subset V_{-2}$. 

\begin{claim}
Define $V'_0 \coloneqq [V_0,V_0]$. We have $V_0=V_0' \oplus \mathbb{R}h$ where $V_0'$ acts on cohomology via derivations.
\end{claim}

\begin{proof}[Proof of the claim]
Proposition \ref{prop:3planes_open} implies that the set $\{(a,b) \in H^2(X,\mathbb{R}) \times H^2(X,\mathbb{R}) \mid a,b \in F(g) \textrm{ for some metric } g \}$ is open. Arguing as in the proof of Lemma \ref{lem:abelian} we see that $V_0$ is generated by the elements $[e_a,f_b]$ with $a,b \in F(g)$ for some metric $g$. If we fix the hyperk\"ahler metric $g$, the elements $[e_a,f_b]$ with $a,b \in F(g)$ generate the Lie algebra $(\mathfrak{g}_g)_0$ and their brackets the Lie subalgebra $(\mathfrak{g}_g)'_0$. Thus, $V'_0$ is generated by the Lie algebras $(\mathfrak{g}_g)'_0$ and their brackets. Since the Lie algebras $(\mathfrak{g}_g)'_0$ act on cohomology via derivations, the same is true for their brackets, hence $V'_0$ acts via derivations. Moreover, from point $(3)$ of Theorem \ref{valg} we get the decomposition $V_0=V_0' + \mathbb{R}h$. Since $\LLV_0$ is reductive (Proposition \ref{LL:ss}) and $h$ is in the center, we get $h \not \in V_0' \subset \LLV'_0$, so the sum is direct.
\end{proof}

Now we show that $[V_0,V_2] \subset V_2$. Since the adjoint action of $h$ gives the grading, it is enough to show that $[V_0',V_2] \subset V_2$. Let $u \in V_0'$ and $e_a \in V_2$. For every $x \in H^2(X,\mathbb{R})$ we have
\begin{equation}\label{eq:derivations}
     [u,e_a](x) = u(a.x) -a.u(x) = u(a).x=e_a(x),
\end{equation}
because $u$ is a derivation. 

The inclusion $[V_0,V_{-2}] \subset V_{-2}$ is more difficult. Let $G'_0 \subset GL(H^*(X,\mathbb{R}))$ be the closed Lie subgroup with Lie algebra $V'_0$. For every $t \in G'_0$ we have $t e_a t^{-1} = e_{t(a)}$ and $t h t^{-1} =h $, by integrating the analogous relations at the level of Lie algebras. Since the third element of an $\mathfrak{sl_2}$-triple is unique, we get that $t f_a t^{-1} = f_{t(a)}$. This implies that the adjoint action of $G'_0$ leaves $V_{-2}$ invariant, hence so does the Lie algebra $V'_0$. 

To summarize, at this point we showed $(1)$ and $(2)$, and also that $\LLV'_0$ acts via derivations. It remains to show that $\LLV'_0 \cong \mathfrak{so}(H^2(X,\mathbb{R}),q)$. 

We begin by defining the map $\LLV'_0 \rightarrow \mathfrak{so}(H^2(X,\mathbb{R}),q)$. For this, we consider the restriction of the action of $\LLV'_0$ to $H^2(X,\mathbb{R})$, and show that it preserves infinitesimally the Beauville--Bogomolov--Fujiki form $q$. We can fix a hyperk\"ahler metric $g$ and check this for $(\mathfrak{g}_g)'_0$, because these Lie subalgebras generate $\LLV'_0$. From Theorem \ref{valg} it is enough to check it for the Weil operators associated to the three complex structures $I,J,K$ induced from $g$. Fix one of them, say $I$; we have to verify that 
\[ q(I\alpha,\beta ) + q(\alpha, I \beta) = 0,\]
for every $\alpha,\beta \in H^2(X,\mathbb{R})$. This follows from a direct verification using the $q$-orthogonal decomposition
\[ H^2(X,\mathbb{R}) = (H^{2,0}(X) \oplus H^{0,2}(X))_{\mathbb{R}} \oplus H^{1,1}(X,\mathbb{R}), \]
induced by the Hodge decomposition with respect to the complex structure $I$. 

To conclude the proof it remains to show that this map is bijective; we begin with the surjectivity. Fix a hyperk\"ahler metric $g$, the image of the Lie algebra $(\mathfrak{g}_g)'_0$ in $\mathfrak{so}(H^2(X,\mathbb{R}),q)$ is generated (as a vector space) by the Weil operators associated to $I,J,K$. Using this, it is easy to see that $(\mathfrak{g}_g)'_0$ kills the $q$-orthogonal complement to the characteristic $3$-plane $F(g)$, and it maps onto $\mathfrak{so}(F(g),q|_{F(g)})$. One can check that varying the metric $g$ the Lie subalgebras $\mathfrak{so}(F(g),q|_{F(g)})$ generate $\mathfrak{so}(H^2(X,\mathbb{R}))$, hence the surjectivity. 

For the injectivity we proceed as follows. Let $SH^2(X,\mathbb{R}) \subset H^*(X,\mathbb{R})$ be the graded subalgebra generated by $H^2(X,\mathbb{R})$; it is a $\LLV$ representation for Corollary \ref{cor:verbrep}. By Lemma \ref{lem:injectivity}, the map $\LLV \rightarrow \mathfrak{gl}(SH^2(X,\mathbb{R}))$ is injective. Since $\LLV'_0$ acts via derivations, the map must be injective already at the level of $H^2(X,\mathbb{R})$.
\end{proof}

\begin{definition}
We define the \textit{Mukai completion} of the quadratic vector space $(H^2(X,\mathbb{Q}),q)$ as the quadratic vector space 
\[ (\Tilde{H}(X,\mathbb{Q}),\Tilde{q}) \coloneqq (H^2(X,\mathbb{Q}),q) \oplus U\]
where $U$ is a two dimensional vector space with quadratic form given by $\begin{pmatrix}
0 & 1 \\
1 & 0 
\end{pmatrix}$
\end{definition}

\begin{cor}
There is a natural isomorphism
\[ \mathfrak{g}_{\mathrm{tot}}(X) \cong \mathfrak{so}(\Tilde{H}(X,\mathbb{Q}),\Tilde{q}). \]
\end{cor}

\begin{proof}
Recall that for a rational quadratic space $(V,q)$ there is an isomorphism 
\begin{align*}
    \bigwedge^2 V &\xrightarrow{\simeq} \mathfrak{so}(V,q), \\
    x \wedge y & \mapsto \frac{1}{2} (q(x,-)y - q(y,-)x)
\end{align*}
The desired isomorphism follows from this, at least at the level of vector spaces. The computations to show that it is in fact an isomorphism of Lie algebras are carried out in \cite[Proposition $2.7$]{GKLR20}.
\end{proof}

\begin{ex}
If $X$ is a $K3$ surface, then the Mukai completion $\tilde{H}(X,\mathbb{Q})$ is the rational cohomology $H^*(X,\mathbb{Q})$ with the usual Mukai pairing. This identification is compatible with the action of $\LLV$. 
\end{ex}

\begin{cor}
The Hodge structure on $H^*(X,\mathbb{R})$ is determined by the Hodge structure on $H^2(X,\mathbb{R})$ and by the action of $\LLV_{2,\mathbb{R}} \cong H^2(X,\mathbb{R})$ on $H^*(X,\mathbb{R})$. 
\end{cor}

\begin{proof}
Let $I,J,K$ be the three complex structures associated to a hyperk\"ahler metric $g$, and assume $I$ is the given one. As recalled before, the commutator $K_{JK}=[e_{J},f_{K}]$ acts like the Weil operator for $I$; hence it recovers the Hodge structure. By definition, it depends only on the classes $[\omega_I],[\omega_K]$ and their action on $H^*(X,\mathbb{R})$. Since the Hodge structure is given by the class of the symplectic form $[\sigma_I]=[\omega_J]+i[\omega_K]$, the conclusion follows.
\end{proof}

Recall that if $\mathfrak{g}$ is a Lie algebra, the \textit{universal enveloping algebra} of $\mathfrak{g}$ is the smallest associative algebra extending the bracket on $\mathfrak{g}$. It is defined as the quotient of the tensor algebra by the elements of the form: 
\[ x \otimes y -y \otimes x -[x,y] \quad x,y \in \mathfrak{g}.\]
In particular, if $\mathfrak{g}$ is abelian, then $U\mathfrak{g}=\mathrm{Sym}^*\mathfrak{g}$. 

\begin{cor}\label{cor:UEA}
There is a natural decomposition 
\[ U\mathfrak{g}_{\mathrm{tot}}(X)=U\mathfrak{g}_{\mathrm{tot}}(X)_{2} \cdot U\mathfrak{g}_{\mathrm{tot}}(X)_0 \cdot U\mathfrak{g}_{\mathrm{tot}}(X)_{-2}.\]
\end{cor}

\section{Primitive decomposition}
In this section, we study the relationship between the actions of $\LLV$ and $\LLV_0$ on $H^*(X,\mathbb{Q})$, where $X$ is a compact hyperk\"ahler manifold of dimension $\textrm{dim}(X)=2n$. The main reference is \cite{LL97:LooijengaLunts}, see also \cite[Theorem $4.4$]{Eyal:monodromy}.

\begin{definition}
Let $V$ be a $\mathfrak{g}_{\mathrm{tot}}(X)$-representation. We define the primitive subspace as: 
\[ \mathrm{Prim}(V)=\{ x \in V \mid (\LLV_{-2}).x=0 \}.\]
\end{definition}

If $V=H^*(X,\mathbb{Q})$ is the standard representation we denote the primitive subspace as $\mathrm{Prim}(X)$.
\begin{remark}
The primitive subspace $\mathrm{Prim}(V)$ is a $\mathfrak{g}_{\mathrm{tot}}(X)_0$-subrepresentation. This follows from the fact that $[\LLV_0,\LLV_{-2}] \subset \LLV_{-2}$.
\end{remark}

\begin{definition}\label{Def:Verbitsky}
The \textit{Verbitsky component} $SH^2(X,\mathbb{Q}) \subseteq H^*(X,\mathbb Q)$ is the graded subalgebra generated by $H^2(X,\mathbb{Q})$. 
\end{definition}

\begin{prop}[{{\cite[Corollary $1.13$ and Corollary $2.3$]{LL97:LooijengaLunts}}}]
The cohomology $H^*(X,\mathbb{Q})$ is generated by $\mathrm{Prim}(X)$ as a $SH^2(X,\mathbb{Q})$-module. Moreover, if $W \subset \mathrm{Prim}(X) $ is a $\LLV_0$ irreducible subrepresentation, then $SH^2(X,\mathbb{Q}).W \subset H^*(X,\mathbb{Q})$ is an irreducible $\LLV$-module. 
\end{prop}

\begin{proof}
Since $\mathfrak{g}_{\mathrm{tot}}(X)$ is semisimple, we can decompose the cohomology in irreducible $\mathfrak{g}_{\mathrm{tot}}(X)$-representations:
\[ H^*(X,\mathbb{Q}) = V_1 \oplus \dots \oplus V_k.\]
The primitive part is compatible with this decomposition, so we get the decomposition 
\[ \mathrm{Prim}(X) =  \mathrm{Prim}(V_1) \oplus \dots \oplus \mathrm{Prim}(V_k),\]
of $\LLV_0$-representations. 

We first want to show that $SH^2(X,\mathbb{Q}).\mathrm{Prim}(V_i)=V_i$. We have
\begin{equation}\label{eq:prim}
    SH^2(X,\mathbb{Q}).\mathrm{Prim}(V_i)=U\LLV_{2}.\mathrm{Prim}(V_i)  
    =U\LLV.\mathrm{Prim}(V_i) \subset V_i,
\end{equation}
where the first equality follows from the fact that $\LLV_{2}$ is abelian, and the second from Corollary \ref{cor:UEA}. Thus $SH^2(X,\mathbb{Q}).\mathrm{Prim}(V_i)$ is a $\mathfrak{g}_{\mathrm{tot}}(X)$ subrepresentation of $V_i$, but $V_i$ is irreducible, so the equality holds. This proves the first part of the proposition. 

To prove the second part it is enough to show that each $\mathrm{Prim}(V_i)$ is irreducible as a $\LLV_0$-representation. Assume it is not and write $\mathrm{Prim}(V_i)=W_1 \oplus W_2$. The identities \eqref{eq:prim} show that acting with $SH^2(X,\mathbb{Q})$ gives a decomposition $V_i=SH^2(X,\mathbb{Q}).W_1 \oplus SH^2(X,\mathbb{Q}).W_2$. Again, this contradicts the fact that $V_i$ is an irreducible $\LLV$-representation.  
\end{proof}

\begin{cor}\label{cor:verbrep}
The Verbitsky component $SH^2(X,\mathbb{Q}) \subset H^*(X,\mathbb{Q})$ is an irreducible $\mathfrak{g}_{\mathrm{tot}}(X)$ subrepresentation.
\end{cor}

\begin{proof}
By definition we have $SH^2(X,\mathbb{Q})=SH^2(X,\mathbb{Q}).H^0(X,\mathbb{Q})$, and $H^0(X,\mathbb{Q}) \subset \mathrm{Prim}(X)$. So it is enough to observe that $H^0(X,\mathbb{Q})$ is preserved by $\LLV_0$, then we conclude by the previous proposition. 
\end{proof}

\begin{lem}\label{lem:injectivity}
The restriction map $\LLV_{\mathbb{R}} \rightarrow \mathfrak{gl}(SH^2(X,\mathbb{R}))$ is injective. 
\end{lem}

\begin{proof}
Let $K \subset \LLV_{\mathbb{R}}$ be the kernel. It is immediate to see that $K \subset \LLV'_0$. The action of $K$ is $0$ on $H^2(X,\mathbb{R})$, so by \eqref{eq:derivations} we get $[K,\LLV_{\mathbb{R},2}]=0$. Taking the Lie group of $K$ and the corresponding adjoint action, we see that $[K,f_a]=0$ for every $a \in H^2(X,\mathbb{R})$ for which $f_a$ is defined. So $K$ has bracket $0$ with $\LLV_{\mathbb{R},2}$ and $\LLV_{\mathbb{R},-2}$, thus also with $\LLV_{\mathbb{R},0}$. Since $\LLV$ is semisimple this implies $K$=0.
\end{proof}

\section{Verbitsky's Theorem}

In this section we give a proof of a result by Verbitsky on the structure of the irreducible component $SH^2(X)$. The argument presented was given by Bogomolov in \cite{Bo96:BogomolovSH}. 

\begin{thm}\label{thm:Ver}
There is a natural isomorphism of algebras and $\LLV_0$-representations: 
\[ SH^2(X,\mathbb{\mathbb{C}}) \cong \mathrm{Sym}^*(H^2(X,\mathbb{C}))/\langle \alpha^{n+1} \mid q(\alpha)=0 \rangle.\]
\end{thm}

The key technical fact is the following lemma from representation theory, of which we omit the proof.

\begin{lem}\label{lem:Vertrick}
Denote by $A$ the graded $\mathbb{C}$-algebra $\mathrm{Sym}^*(H^2(X,\mathbb{C}))/\langle \alpha^{n+1} \mid q(\alpha)=0 \rangle$. Then we have:
\begin{enumerate}
    \item $A_{2n} \cong \mathbb{C}$.
    \item The multiplication map $A_{k} \times A_{2n-k} \rightarrow A_{2n}$ induces a perfect pairing.
\end{enumerate}
\end{lem}

\begin{proof}[Proof of the theorem]
From the Local Torelli Theorem we have that $\alpha^{n+1}=0$ for an open subset of the quadric $\{ \alpha \in H^2(X,\mathbb{C}) \mid q(\alpha)=0  \}$. Since the condition $\alpha^{n+1}=0$ is Zariski closed, we get that it holds for the entire quadric. Consider the multiplication map \[\mathrm{Sym}^*(H^2(X,\mathbb{C})) \rightarrow SH^2(X,\mathbb{C}).\] The kernel contains $\{ \alpha^{n+1} \mid q(\alpha)=0 \}$, hence it factors via the ring $A$. It is an algebra homomorphism by construction, and a map of $\LLV_0$-representations because $\LLV'_0$ acts via derivations. 

The induced map $A \rightarrow SH^2(X,\mathbb{C})$ is surjective by construction. If it were not injective, by the above lemma, the kernel would contain $A_{2n}$. But this is impossible, because in top degree the map $A_{2n} \rightarrow H^{4n}(X,\mathbb{C})$ is non-zero. Indeed if $\sigma$ is a holomorphic symplectic form, the form $(\sigma +\overline{\sigma})^{2n}$ is non-zero.
\end{proof}

\begin{cor}
There are natural isomorphisms defined over $\mathbb{Q}$
\[ 
SH^2(X,\mathbb{Q})_{2k} \cong 
\begin{cases}
\mathrm{Sym}^{k} H^2(X,\mathbb{Q}) \textrm{ if } k\leq n, \\
\mathrm{Sym}^{2n-k} H^2(X,\mathbb{Q}) \textrm{ if } n < k \leq 2n.
\end{cases}
\]
\end{cor}

\begin{proof}
From Theorem \ref{thm:Ver} it follows that the properties $(1)$ and $(2)$ in Lemma \ref{lem:Vertrick} hold for $SH^2(X,\mathbb{C})$. Up to changing the isomorphism $SH^2(X,\mathbb{C})_{2n} \cong \mathbb{C}$, they also hold for $SH^2(X,\mathbb{Q})$. The multiplication map $\mathrm{Sym}^{k} H^2(X,\mathbb{Q}) \rightarrow SH^2(X,\mathbb{Q})_{2k}$ is an isomorphism if $k \leq n$, because it is so over $\mathbb{C}$. If $k>n$ we have \[ SH^2(X,\mathbb{Q})_{2k} \cong SH^2(X,\mathbb{Q})^*_{4n-2k} \cong \mathrm{Sym}^{2n-k} H^2(X,\mathbb{Q})^* \cong \mathrm{Sym}^{2n-k} H^2(X,\mathbb{Q}),\]
where the last equality is due to the Beauville--Bogomolov--Fujiki form. 
\end{proof}

\begin{ex}
If $X$ is of $\mathrm{K3}^{[2]}$-type, for dimensional reasons, the Verbitsky component $SH(X)$ is the only irreducible component in the cohomology. For higher values of $n$ the decomposition of $H^*(X,\mathbb{Q})$ in irreducible components is described in \cite{GKLR20}, for more details on this see \cite{notes:OS}.
\end{ex}

\section{Spin action}
In this section we study how the action of $\mathfrak{so}(H^2(X,\mathbb{Q}),q)$ integrates to an action of the simply connected algebraic group $\Spin$. Recall that there is an exact sequence of algebraic groups 
\[ 1 \rightarrow {\pm 1} \rightarrow \Spin \rightarrow \SO \rightarrow 1.  \]
For more information see \cite{notes:B} and \cite{notes:OS}.

\begin{prop}[{{\cite[Theorem $4.4$]{Eyal:monodromy},\cite{VerbitskyMirror}}}] 
The action of $\mathfrak{so}(H^2(X,\mathbb{Q}),q)$ on $H^*(X,\mathbb{Q})$ integrates to an action of the algebraic group $\Spin$ via ring isomorphisms. On the even cohomology it induces an action of $\SO$. 
\end{prop}

\begin{proof}
The first part of the statement is clear: we can always lift the action because the algebraic group $\Spin$ is simply connected. The group $\Spin$ acts via ring isomorphisms because the Lie algebra acts via derivations. 

To show the second part of the statement we proceed as follows. Fix a hyperk\"ahler metric $g$ and a compatible complex structure $I$. The Weil operator with respect to $I$ is contained in $(\mathfrak{g}_g)'_0 \cong \mathfrak{so}(H^2(X,\mathbb{Q}))$. The exponential $\exp(\pi I) \in \Spin$ acts on the $(p,q)$ part of $H^k(X,\mathbb{C})$ as multiplication by $e^{i(p-q)\pi}$, which is just multiplication by $(-1)^k$. In particular, on $H^2(X,\mathbb{Q})$ it acts as the identity, so $\exp(\pi I)=-1 \in \Spin$. We have also shown that $-1 \in \Spin$ acts on $H^k(X,\mathbb{Q})$ as $(-1)^k$, which means that the action on even cohomology factors through $\SO$.
\end{proof}

\subsection*{Acknowledgments}
I want to thank Fabrizio Anella, Olivier Debarre, Daniel Huybrechts and Mauro Varesco for reading these notes and for their comments. 

\bibliography{HyperK}
\bibliographystyle{alphaspecial}
\end{document}